\documentclass{amsart}

\copyrightinfo{2006}{American Mathematical Society}

\newtheorem{theorem}{Theorem}[section]
\newtheorem*{theorem*}{Theorem}
\newtheorem{lemma}[theorem]{Lemma}
\newtheorem{proposition}[theorem]{Proposition}

\theoremstyle{remark}
\newtheorem{remark}[theorem]{Remark}

\theoremstyle{definition}
\newtheorem*{notation}{Notation}

\numberwithin{equation}{section}

\newcommand{\Pd}{\mathcal{P}_d}
\newcommand{\re}{\mathbb{R}}

\begin{document}

\title[The Stein-Wainger Oscillatory integral]{A sharp bound for the Stein-Wainger oscillatory integral}

 \author[Ioannis R. Parissis]{Ioannis R. Parissis}

\address{Department of Mathematics, University of Crete, Knossos Avenue 71409, Iraklio--Crete, Greece}

\email{ypar@math.uoc.gr}
\thanks{}

\subjclass[2000]{Primary 42A50; Secondary 42A45}



\begin{abstract}
Let $\Pd$ denote the space of all real polynomials of degree at most $d$. 
It is an old result of Stein and Wainger \cite{SWA} that 
$$\sup_ {P\in\Pd} \bigg|p.v.\int_\re {e^{iP(t)}\frac{dt}{t}}\bigg|\leq C_d$$ 
for some constant $C_d$ depending only on $d$. On the other hand, Carbery, Wainger and Wright in \cite{CWAW} claim that the true order of magnitude of the above principal value integral is $\log d$. We prove that 
$$ \sup_ {P\in\Pd}\bigg|p.v.\int_\re {e^{iP(t)}\frac{dt}{t}}\bigg|\sim \log{d}.$$ 
\end{abstract}

\maketitle

\section{Introduction}
Let $\Pd$ be the vector space of all real polynomials of degree at most $d$ in $\re$. For $P\in\Pd$ we consider the principal value integral 
$$I(P)=\bigg|p.v.\int_\re {e^{iP(t)}\frac{dt}{t}}\bigg|.$$ We wish to estimate the quantity $I(P)$ by a constant $C(d)$ depending only on the degree of the polynomial $d$. This amounts to estimating the integral 
$$I_{(\epsilon,R)}(P)=\bigg|\int_{\epsilon \leq |t| \leq R} {e^{iP(t)}\frac{dt}{t}}\bigg| $$
by some constant $C(d)$ independent of $\epsilon, R$ and $P$.

This problem is quite old and in fact has been answered some thirty years ago by Stein and Wainger in \cite{SWA} and \cite{WA}. They showed that the quantity $I(P)$ is bounded by a constant $C_d$ depending only on $d$. Their proof is very simple and uses a combination of induction and Van der Corput's lemma. Let us recall the latter since we'll also be using it in what follows.
\pagebreak
\begin{proposition}[van der Corput]\label{prop1}
Let $\phi:[a,b]\rightarrow\re$ be a $C^k$ function and suppose that $|\phi^{(k)}(t)|\geq 1$ for some $k\geq 1$ and all $t\in[a,b]$. If $k=1$ suppose in addition that $\phi^{\prime}$ is monotonic. Then, for every $\lambda \in \re$,

$$\bigg| \int _a ^b e^{i\lambda \phi(x)} dx \bigg | \leq \frac{Ck}{|\lambda|^{\frac{1}{k}}}$$ 
where $C$ is an absolute constant independent of a,b,k and $\phi$.
\end{proposition}

For a proof of this very well known result with $C k$ replaced by $C_k$ see for example \cite{S}. A proof that the constant $C_k$ can be taken to be linear in $k$ can be found in \cite{AKC}.

On the other hand, Carbery, Wainger and Wright have conjectured in \cite{CWAW} that the true order of magnitude of the principal value integral is $\log d$. The main result of this paper is the proof of this conjecture. This is the content of:

\begin{theorem*}
There exist two absolute positive constants $c_1$ and $c_2$ such that
$$c_1 \log d \leq  \sup_ {P\in\Pd} \bigg|p.v.\int_\re {e^{iP(x)}\frac{dx}{x}}\bigg| \leq c_2 \log d .$$
\end{theorem*}

\begin{remark} Suppose that $K$ is a $-n$ homogeneous function on $\re^n$, odd and integrable on the unit sphere. Then, by the one-dimensional result, we trivially get that there is an absolute positive constant $c$, such that:
$$\bigg|p.v.\int_{\re^n} e^{iP(x)}K(x)dx \bigg| \leq c  \|K\|_{L^1 (S^{n-1})}\log d, $$ 
for every polynomial $P$ on $\re^n$, of degree at most $d$.
\end{remark}
\begin{notation} We will use the letter $c$ to denote an absolute positive constant which might change even in the same line of text. Also, the notation $A\sim B$ means that there exist absolute positive constants $c_1$ and $c_2$ such that $c_1 B\leq A \leq c_2 B$. 
\end{notation}
\section{Aknowledgements}
I would like to thank James Wright for bringing this problem to our attention and for many helpful discussions. I would also like to thank Mihalis Papadimitrakis from the University of Crete, my thesis supervisor, for his constant support.
\section{The lower bound in the Theorem}
In this section we will construct a real polynomial $P$ of degree at most $d$ such that the inequality 
\begin{equation}\label{eq1}
I(P)=\bigg|p.v.\int_\re {e^{iP(t)}\frac{dt}{t}}\bigg| \geq c \log d
\end{equation}
holds. The general plan of the construction is as follows. We will first construct a \emph{function} $f$ (which will not be a polynomial) such that $I(f)\geq c \log n$. We will then construct a polynomial $P$ of degree $d=2n^2-1$ that approximates the function $f$ in a way that $|I(f)-I(P)|$ is small (small means $o(\log n)$ here). Since $\log n \sim \log d$ this will yield our result.			 

\begin{lemma}\label{lm1} For $n$ a large positive integer, let $f(t)$ be the continuous function which is equal to $1$ for $\frac{1}{n}\leq t \leq 1-\frac{1}{n}$, equal to $-1$ for  $-1+\frac{1}{n}\leq t \leq -\frac{1}{n}$, equal to $0$ for $|t|\geq1$ and linear in each interval $[-1,-1+\frac{1}{n}]$, $[-\frac{1}{n},\frac{1}{n}]$ and $[1-\frac{1}{n},1]$. 
Then,
\begin{equation}\label{eq2}
I(f)=\bigg|p.v.\int_\re {e^{if(t)}\frac{dt}{t}}\bigg| \geq c \log n.
\end{equation}
\end{lemma}

\begin{proof} The proof is more or less straightforward. 
\begin{eqnarray*}
I(f)&=&2\bigg|\int_0 ^1 \frac{\sin{f(t)}}{t}dt\bigg|\\&\geq& 2\bigg|\int_\frac{1}{n} ^{1-\frac{1}{n}} \frac{\sin{f(t)}}{t}dt\bigg| -
2\bigg|\int_0 ^\frac{1}{n} \frac{\sin{f(t)}}{t}dt\bigg|-2\bigg|\int_{1-\frac{1}{n}} ^1 \frac{\sin{f(t)}}{t}dt\bigg| \\
&\geq&2\sin1 \log(n-1)-2\int_0 ^\frac{1}{n}\frac{f(t)}{t}dt-2\int_{1-\frac{1}{n}} ^1\frac{f(t)}{t}dt\\
&=&2\sin1 \log(n-1)-2-2n\log\frac{n}{n-1}+2\\&\geq& 2\sin1 \log(n-1)-4\,\,\, \geq\,\,\, c\log n.
\end{eqnarray*}
\end{proof}

We now want to construct a polynomial which approximates the function $f$. We will do so by convolving the function $f$ with a "polynomial approximation to the identity". To be more specific, for $k\in\mathbb{N}$ and $x\in\re$ define the function
\begin{equation}\label{eq3}
\phi_k(x)  = c_k\bigg(1-\frac{x^2}{4}\bigg)^{k^2}
\end{equation}
where the constant $c_k$ is defined by means of the normalization 
\begin{equation}\label{eq4}
\int_{-2} ^2 \phi_k(x)dx =1 .
\end{equation}
Observe that
$$1=c_k\int_{-2} ^2 \bigg(1-\frac{x^2}{4}\bigg)^{k^2} dx=4c_k\int_0 ^1 (1-x^2)^{k^2} dx=2c_kB\bigg(\frac{1}{2},k^2+1\bigg),$$
where $B(\cdot,\cdot)$ is the beta function. Using standard estimates for the beta function we see that $c_k\sim k$.

Define, next, the functions $P_k$ in $\re$ as
\begin{equation}\label{eq5}
P_k(t)=\int_{-1} ^1 f(x)\phi_k(t-x)dx ,
\end{equation}
where $f$ is the function of Lemma \ref{lm1}. It is clear that the functions $P_k$ are \emph{polynomials} of degree at most $2k^2$. The following lemma deals with some technical issues concerning the polynomials $P_k$.
\medskip
\pagebreak
\begin{lemma}\label{lm2} Let $P_k$ be defined as in (\ref{eq5}) above.
\item(i) $P_k$ is an odd polynomial of degree $2k^2-1$ with leading coefficient $$a_k=(-1)^{k^2+1}\frac{2c_k k^2}{4^{k^2}}\bigg(1-\frac{1}{n}\bigg).$$ That is $$P_k(t)=a_k t^{2k^2-1} + \cdots \ .$$
\item(ii) As a consequence of (i) we have for all t $$|P^{(2k^2-1)}_k(t)|\geq c (2k^2-1)! \frac{k^3}{4^{k^2}}.$$	 
\item(iii) For $t\in
[-1,1]$ we have $$P_k(t)=\int_{0} ^2 \big(f(t+x)+f(t-x)\big)\phi_k(x) dx . $$
\end{lemma}
\smallskip\begin{proof}(i) Using (\ref{eq5}) we have
\begin{eqnarray*}
P_k(-t)&=&\int_{-1} ^1 f(x)\phi_k(-t-x)dx\,\,\, =\,\,\, \int_{-1} ^1 f(x)\phi_k(t+x)dx\\
&=&\int_{-1} ^{1} f(-x)\phi_k(t-x)dx\,\,\, =\,\,\, -P_k(t).
\end{eqnarray*}

Next, from (\ref{eq5}) we have that
\begin{eqnarray*}
P_k(t)&=&c_k\int_{-1} ^{1} f(x) \sum_{m=0} ^{k^2} \binom{k^2}{m} \bigg( -\frac{(t-x)^2}{4} \bigg)^m dx\\
&=&c_k \sum_{m=0} ^{k^2} \binom{k^2}{m} \frac{(-1)^m}{4^m}\int_{-1}^1 f(x)	(t-x)^{2m} dx \\
&=&c_k \frac{(-1)^{k^2}}{4^{k^2}}\int_{-1} ^1 f(x)(x-t)^{2k^2} dx\\&+&
c_k \sum_{m=0} ^{k^2-1} \binom{k^2}{m} \frac{(-1)^m}{4^m}\int_{-1}^1 f(x)	(t-x)^{2m} dx .
\end{eqnarray*}
It is now easy to see that the two highest order terms come from the first summand in the above formula. Therefore,
\begin{eqnarray*}
P_k(t)&=&c_k \frac{(-1)^{3k^2}}{4^{k^2}}\int_{-1} ^1 f(x)dx\ t^{2k^2}-c_k \frac{(-1)^{k^2}2k^2}{4^{k^2}}\int_{-1} ^1 f(x)x dx\  t^{2k^2-1}+\cdot\cdot\ \\
&=&(-1)^{k^2+1}\frac{2c_k k^2}{4^{k^2}}\bigg(1-\frac{1}{n}\bigg)t^{2k^2-1} +\cdots.
\end{eqnarray*}
(ii) We just use the result of (i) and that $c_k\sim k$.

\noindent(iii) Fix a $t\in[-1,1]$. Then, 
\begin{eqnarray*}
\int_{-2} ^2 f(t-x)\phi_k(x)dx&=&\int_\re f(t-x) \phi_k(x) \chi_{[-2,2]}(x)dx \\
&=&\int_{-1}^1 f(x)\phi_k(t-x) \chi_{[-2,2]}(t-x)dx\\
&=&\int_{-1}^1 f(x)\phi_k(t-x) dx\\&=&P_k(t).
\end{eqnarray*}
However, since $\phi_k$ is even,
$$P_k(t)=\int_{-2} ^2 f(t-x)\phi_k(x)dx=\int_{0} ^2 \big(f(t+x)+f(t-x)\big)\phi_k(x)dx.$$
\end{proof}

We are now ready to prove the lower bound for $I(P)$.

\begin{proposition}\label{prop2} Let $P_n$ be the polynomial defined in (\ref{eq5}) where $n$ is the large positive integer used to define the function $f$ in Lemma \ref{lm1}. Then $P_n$ is a polynomial of degree $d=2n^2-1$ and
$$I(P_n)=\bigg|p.v.\int_\re e^{iP_n(t)}\frac{dt}{t}\bigg| \geq c \log d .$$
\end{proposition}
\smallskip
\begin{proof} Since $P_n$ is odd,
$$I(P_n)=2 \bigg|\int _0 ^{+\infty}\frac{\sin{P_n(t)}}{t}dt\bigg| ,$$
and it suffices to show that for all $R\geq1$
\begin{equation}\label{eq6}
\bigg|\int _0 ^R\frac{\sin{P_n(t)}}{t}dt\bigg| \geq c \log d\sim c\log n.
\end{equation}
By part (ii) of Lemma \ref{lm2} and a standard application of Proposition \ref{prop1} (Van der Corput) we see that
$$\bigg|\int _1 ^R\frac{\sin{P_n(t)}}{t}dt\bigg| \leq c$$
for all $R\geq 1$. As a result, the proof will be complete if we show that
\begin{equation}\label{eq7}
I_1(P_n)=\bigg|\int _0 ^1\frac{\sin{P_n(t)}}{t}dt\bigg| \geq c\log n.
\end{equation}
Using Lemma \ref{lm1} and the triangle inequality we get 
\begin{equation}\label{eq8}
I_1(P_n)\geq c \log n -|I_1(P_n)-I(f)|
\end{equation}
and, in order to show (\ref{eq7}), it suffices to show that
\begin{equation}\label{eq9}
|I_1(P_n)-I(f)|=o(\log n).
\end{equation}

We have that
\begin{eqnarray*}
|I_1(P_n)-I(f)|&=& \bigg|\int_0 ^1 \frac{\sin P_n(t)-\sin f(t)}{t}dt   \bigg|\\ 
&\leq& \int_0 ^1 \frac{| P_n(t)-f(t)|}{t}dt. 
\end{eqnarray*}
Using part (iii) of Lemma \ref{lm2} and (\ref{eq4}), we get
$$|P_n(t)-f(t)|\leq \int _0 ^2 |f(t+x)+f(t-x)-2f(t)|\phi_n(x) dx $$
for $0\leq t \leq 1$.
Hence
$$|I_1(P_n)-I(f)|\leq \int_0 ^2 \int _0 ^1 \frac{|f(t+x)+f(t-x)-2f(t)|}{t}\  dt \ \phi_n(x) dx .$$
Now, the desired result, condition (\ref{eq9}), is the content of the following lemma.\end{proof}
\begin{lemma}\label{lm3} Let $A(x,t)=|f(t+x)+f(t-x)-2f(t)|$. Then,
$$ \int_0 ^2 \int _0 ^1 \frac{A(x,t)}{t} \ dt  \ \phi_n(x) dx = o(\log n).$$
\end{lemma}
\smallskip
\begin{proof} Firstly, it is not difficult to establish that
\begin{eqnarray}
A(x,t)&\leq& 4 \min(nx,nt,1) \label{eq10}\\
A(x,t)&=&0,\qquad \mbox{when} \ \ \frac{1}{n}\leq t-x \leq t+x \leq 1-\frac{1}{n} \ .\label{eq11}
\end{eqnarray}
Indeed, 
\begin{eqnarray*}
A(x,t)&\leq&|f(t+x)-f(t)|+|f(t-x)-f(t)| \\
&\leq& nx +nx\,\,\, \leq\,\,\, 2nx .
\end{eqnarray*}
On the other hand, 
\begin{eqnarray*}
A(x,t)&=&|f(t+x)-f(x)+f(t-x)-f(-x)-2f(t)| \\
&\leq& |f(t+x)-f(x)|+|f(t-x)-f(-x)|+2|f(t)| \\
&\leq& nt+nt+2nt\,\,\, =\,\,\, 4nt .
\end{eqnarray*}
Inequality (\ref{eq10}) now follows by the fact that $|f|$ is bounded by 1 and (\ref{eq11}) is trivial to prove.

We split the integral $\int_0 ^2 \int_0 ^1 \cdots dtdx$ into seven integrals:
\begin{eqnarray*}
&&\int_0^2 \int _\frac{1}{2} ^1 \cdots dtdx +\int_0^\frac{1}{n}\int_0^x \cdots dtdx+\int_\frac{1}{n}^2\int_0^\frac{1}{n} \cdots dtdx+\int_0^\frac{1}{n}\int_x ^{x+\frac{1}{n}} \cdots dtdx\ \\ &&+\int_0 ^{\frac{1}{2}-\frac{1}{n}}\int_{x+\frac{1}{n}} ^{\frac{1}{2}} \cdots dtdx+\int_\frac{1}{n} ^{\frac{1}{2}-\frac{1}{n}} \int_\frac{1}{n} ^{x+\frac{1}{n}} \cdots dtdx+ \int_{\frac{1}{2}-\frac{1}{n}} ^2 \int_\frac{1}{n} ^\frac{1}{2} \cdots dtdx.
\end{eqnarray*}
We estimate each of the seven integrals separately.
$$\int_{0} ^2 \int _\frac{1}{2} ^1 \frac{A(x,t)}{t} dt  \phi_n(x) dx\,\,\, \leq\,\,\, 4\log 2\int_0 ^2 \phi_n(x) dx\,\,\, =\,\,\, 2 \log 2.$$

\begin{eqnarray*}\int_0 ^\frac{1}{n} \int_0 ^x \frac{A(x,t)}{t}dt \phi_n(x)dx&\leq& \int_0 ^\frac{1}{n} \int_0 ^x \frac{4nt}{t}dt \phi_n(x)dx \\ &=& \int_0 ^\frac{1}{n} 4nx \phi_n(x)dx\,\,\, \leq\,\,\, 2 .
\end{eqnarray*}
\begin{eqnarray*}\int_\frac{1}{n} ^2 \int_0 ^\frac{1}{n} \frac{A(x,t)}{t}dt \phi_n(x)dx&\leq& \int_\frac{1}{n} ^2 \int_0 ^\frac{1}{n} \frac{4nt}{t}dt \phi_n(x)dx \\ &=& \int_0 ^\frac{1}{n} 4 \phi_n(x)dx\,\,\, \leq\,\,\, 2 .
\end{eqnarray*}

\begin{eqnarray*}
\int _0 ^\frac{1}{n} \int _{x} ^{x+\frac{1}{n}}\frac{A(x,t)}{t}dt\phi_n(x)dx&\leq&\int _0 ^\frac{1}{n} \int _{x} ^{x+\frac{1}{n}}\frac{4nx}{t}dt\phi_n(x)dx \\
&=&\int _0 ^\frac{1}{n} 4nx\log\bigg(1+\frac{1}{nx}\bigg)\phi_n(x)dx\,\,\, \leq\,\,\, 2.
\end{eqnarray*}

For $\int_0 ^{\frac{1}{2}-\frac{1}{n}}\int_{x+\frac{1}{n}} ^\frac{1}{2}$ we have $\frac{1}{n}\leq t-x\leq t+x \leq 1-\frac{1}{n}$ and, by (\ref{eq11}), $A(x,t)=0$. Hence
$$\int_0 ^{\frac{1}{2}-\frac{1}{n}}\int_{x+\frac{1}{n}} ^\frac{1}{2}\frac{A(x,t)}{t}dt\phi_n(x)dx=0.$$
Next
\begin{eqnarray*}
\int _\frac{1}{n} ^{\frac{1}{2}-\frac{1}{n} }\int _\frac{1}{n}^{x+\frac{1}{n}}\frac{A(x,t)}{t}dt\phi_n(x)dx &\leq&
\int _\frac{1}{n} ^{\frac{1}{2}-\frac{1}{n} }\int _\frac{1}{n}^{x+\frac{1}{n}}\frac{4}{t}dt\phi_n(x)dx\\
&\leq&4\int _\frac{1}{n} ^1 \log(nx+1)\phi_n(x)dx.
\end{eqnarray*}
Now, fix some $\alpha\in(0,1)$. Write 
\begin{eqnarray*}
\int _\frac{1}{n} ^1 \log(nx+1)\phi_n(x)dx&=&\int_\frac{1}{n} ^\frac{1}{n^\alpha}\cdots\, dx+\int_\frac{1}{n^\alpha} ^1\cdots \, dx\\
&\leq&\frac{\log(n^{1-\alpha}+1)}{2} + c_n\log (n+1) \int_\frac{1}{n^\alpha} ^1\bigg(1-\frac{x^2}{4} \bigg)^{n^2}dx\\
&\leq&\frac{\log(n^{1-\alpha}+1)}{2}+ c n \log (n+1) \ e^{-\frac{1}{4}n^{2(1-\alpha)}}.
\end{eqnarray*}
Therefore,
$$\limsup_{n\rightarrow \infty}\frac{\int _\frac{1}{n} ^1 \log(nx+1)\phi_n(x)dx}{\log n}\leq \frac{1-\alpha}{2}$$
and, since $\alpha$ is arbitrary in $(0,1)$,$$\int _\frac{1}{n} ^{\frac{1}{2}-\frac{1}{n} }\int _\frac{1}{n}^{x+\frac{1}{n}}\frac{A(x,t)}{t}dt\phi_n(x)dx=o(\log n).$$
Finally,
\begin{eqnarray*}
\int _{\frac{1}{2}-\frac{1}{n}} ^2 \int _\frac{1}{n}^{\frac{1}{2}}\frac{A(x,t)}{t}dt\phi_n(x)dx &\leq&
\int _{\frac{1}{2}-\frac{1}{n}} ^2 \int _\frac{1}{n}^{\frac{1}{2}}\frac{4}{t}dt\phi_n(x)dx  \\
&\leq&4\log\frac{n}{2}\  c_n \int _{\frac{1}{2}-\frac{1}{n}} ^2 \bigg(1-\frac{x^2}{4} \bigg)^{n^2}dx\\
&\leq&cn\log n e^{-\frac{1}{16}n^2}\,\,\, =\,\,\, o(1).
\end{eqnarray*}

\end{proof}
\section{The upper bound in the Theorem}

We set 
\begin{equation}\label{eq12}
K_d=\sup_{P\in\Pd,\epsilon,R}\bigg|\int_{\epsilon \leq |t| \leq R} {e^{iP(t)}\frac{dt}{t}} \bigg|.
\end{equation}
We take any polynomial $P$, of degree at most $d$, which we can assume has no constant term, that is, $P(0)=0$. We set $k=[\frac{d}{2}]$ and we write
\begin{eqnarray*}P(t)&=&a_1t+a_2t^2+\cdots+a_k t^k+a_{k+1}t^{k+1}+\cdots+a_dt^d \\
                     &=&Q(t)+R(t),
\end{eqnarray*}
where $Q(t)=a_1t+a_2t^2+\cdots+a_k t^k$ and $R(t)=a_{k+1}t^{k+1}+\cdots+a_dt^d$. Let $|a_l|=max_{k+1\leq j\leq d}|a_j	|$ for some $k+1\leq l \leq d$. By a change of variables in the integral in (\ref{eq12}) we can assume that $|a_l|=1$ and thus that $|a_j|\leq 1$ for every $k+1\leq j \leq d$. Now split the integral in (\ref{eq12}) in two parts as follows
\begin{eqnarray}\label{eq13}
\bigg|\int_{\epsilon \leq |t| \leq R} {e^{iP(t)}\frac{dt}{t}}\bigg| &\leq& \bigg|\int _{\epsilon\leq |t| \leq 1}{e^{iP(t)}\frac{dt}{t}}\bigg| + \bigg|\int _{1\leq |t| \leq R}{e^{iP(t)}\frac{dt}{t}}\bigg|\\ & = & I_1+I_2. \nonumber
\end{eqnarray}
For $I_1$ we have that

\begin{eqnarray*}
I_1 &\leq&  \bigg| \int_{\epsilon \leq |t| \leq 1} \big[e^{iP(t)}-e^{iQ(t)}\big]\frac{dt}{t}\bigg|+\bigg|\int_{\epsilon \leq |t| \leq 1} {e^{iQ(t)}\frac{dt}{t}}\bigg| \\
&\leq& \int_{\epsilon \leq |t| \leq 1} \big|e^{iP(t)}-e^{iQ(t)}\big|\frac{dt}{t}+K_{[\frac{d}{2}]} \\
&\leq &\int_{0 \leq |t| \leq 1} \frac{|R(t)|}{t}dt+K_{[\frac{d}{2}]}\\
&\leq & 2\sum_{j=k+1}^d \frac{|a_j|}{j}+K_{[\frac{d}{2}]}\,\,\, \leq\,\,\, \sum_{j=k+1}^d \frac{1}{j}+K_{[\frac{d}{2}]}\,\,\, \leq\,\,\, c +K_{[\frac{d}{2}]}.
\end{eqnarray*}

For the second integral in (\ref{eq13}) we have that
\begin{eqnarray*}
I_2\leq  \bigg|\int _{1\leq t \leq R}{e^{iP(t)}\frac{dt}{t}}\bigg|  + \bigg|\int _{-R\leq t \leq -1}{e^{iP(t)}\frac{dt}{t}}\bigg|=I_2 ^+ +  I_2 ^- .
\end{eqnarray*}
For some $\alpha >0$ to be defined later split $I_2 ^+$ into two parts as follows:
\begin{eqnarray*}
I_2 ^+ \leq \int_{\{t\in[1,+\infty):|P^\prime(t)|\leq \alpha\}}  \frac{dt}{t} + \bigg|\int _{\{t\in[1,R]:|P^\prime(t)| > \alpha \}}  e^{iP(t)} \frac{dt}{t}\bigg|.
\end{eqnarray*}

Since $\{t\in[1,R]:|P^\prime(t)| > \alpha\}$ consists of at most $O(d)$ intervals where $P^{\prime}$ is monotonic, using Proposition 1 we get the bound
$$\bigg|\int _{\{t\in[1,R]:|P^\prime(t)| > \alpha\} }e^{iP(t)}\frac{dt}{t}\bigg|\leq c \frac{d}{\alpha}.$$

For the logarithmic measure of  the set $\{t\in[1,+\infty):|P^\prime(t)| \leq \alpha\}$, observe that
\begin{eqnarray*}
\int_{\{t\in[1,+\infty):|P^\prime(t)|\leq \alpha\}}\frac{dt}{t}&\leq& \sum_{m=0}^\infty \int_{\{t\in[2^m,2^{m+1}]:|P^\prime(t)|\leq \alpha\}}\frac{dt}{t}\\&=&
\sum_{m=0}^\infty \int_{\{2^mt\in[2^m,2^{m+1}]:|P^\prime(2^mt)|\leq \alpha\}}\frac{dt}{t} \\
&=&\sum_{m=0}^\infty \int_{2^m\{t\in[1,2]:|P^\prime(2^mt)|\leq \alpha\}}\frac{dt}{t}\\& =&\sum_{m=0}^\infty \int_{\{t\in[1,2]:|P^\prime(2^mt)|\leq \alpha\}}\frac{dt}{t}.
\end{eqnarray*} 
We have thus showed that
\begin{eqnarray}\label{eq14}
\int_{\{t\in[1,+\infty):|P^\prime(t)| \leq \alpha\}}\frac{dt}{t}\leq  \sum_{m=0}^\infty \ |\{t\in[1,2]:|P^\prime(2^mt)|\leq \alpha\}|.
\end{eqnarray}

In order to finish the proof we need a suitable estimate for the sublevel set of a polynomial. This is the content of the following lemma.

\begin{lemma}[Vinogradov]\label{lm4} Let $h(t)=b_0+b_1 t+\cdots+b_n t^n$ be a real polynomial of degree n. Then,
$$|\{t\in[1,2]:|h(t)|\leq \alpha\}| \leq c \bigg(\frac{\alpha}{\max_{0\leq k \leq n} |b_k|}\bigg)^\frac{1}{n}.$$
\end{lemma}

This Lemma is due to Vinogradov \cite{V}. We postpone the proof of Lemma \ref{lm4} until after the end of the proof of the upper bound.
 
\bigskip
Consider the polynomial $P^\prime(2^mt)$ with coefficients $ja_j2^{m(j-1)}$, $1\leq j\leq d$. Clearly, 
$\max_{1\leq j\leq d}|ja_j2^{m(j-1)}|\geq |la_l2^{m(l-1)}|\geq( [\frac{d}{2}]+1)2^{m[\frac{d}{2}]} $. Using Lemma 4 and (\ref{eq14}), we get

$$\int_{\{t\in[1,+\infty):|P^\prime(t)| \leq \alpha\}}\frac{dt}{t}\leq c\alpha^\frac{1}{d-1} \sum_{m=0}^\infty \bigg(\frac{1}{([\frac{d}{2}]+1)2^{m[\frac{d}{2}]}}\bigg)^\frac{1}{d-1}\leq c \alpha^\frac{1}{d-1} .$$

Obviously, a similar estimate holds for $I_2 ^-$. Summing up the estimates we get
\begin{eqnarray*}
\bigg|\int_{\epsilon \leq |t| \leq R} {e^{iP(t)}\frac{dt}{t}}\bigg|\leq c+c\frac{d}{\alpha}+c\alpha^\frac{1}{d-1}+K_{[\frac{d}{2}]}.
\end{eqnarray*}
Optimizing in $\alpha$ we get that
\begin{eqnarray}\label{eq15}\bigg|\int_{\epsilon \leq |t| \leq R} {e^{iP(t)}\frac{dt}{t}}\bigg|\leq c +K_{[\frac{d}{2}]}
\end{eqnarray}
and hence 
$$K_d\leq c+ K_{[\frac{d}{2}]}.$$
In particular we have
$$K_{2^n}\leq c +K_{2^{n-1}} .$$
Using induction on $n$ we get that $K_{2^n} \leq c n$. It is now trivial to show the inequality for general d. Indeed, if $2^{n-1} < d \leq 2^n$ then $K_d\leq K_{2^n}\leq c n \leq c \log d$.

For the sake of completeness we give the proof of Lemma \ref{lm4}.
\medskip

\begin{proof}[Proof of Lemma \ref{lm4}]The set $E_\alpha =\{t\in[1,2]:|h(t)|\leq \alpha\}$ is a union of intervals. We slide them together to form a single interval $I$ of length $|E_\alpha|$ and pick $n+1$ equally spaced points  in $I$. If we slide the intervals back to their original position we end up with $n+1$ points  $\,x_0, x_1, x_2, \ldots, x_n\in E_\alpha$ which satisfy 
\begin{eqnarray}\label{eq16}
|x_j-x_k|\geq |E_\alpha|\frac{|j-k|}{n}.
\end{eqnarray}
The Lagrange polynomial which interpolates the values $h(x_0)$, $h(x_1)$,\ldots, $h(x_n)$ coincides with $h(x)$:
$$h(x)=\sum_{j=0} ^nh(x_j)\frac{(x-x_0)(x-x_1)\cdots(x-x_{j-1})(x-x_{j+1})\cdots(x-x_n)}{(x_j-x_0)(x_j-x_1)\cdots(x_j-x_{j-1})(x_j-x_{j+1})\cdots(x_j-x_n)} .$$
Therefore we get for the coefficients of $h$ that  
$$b_k=\sum_{j=0}^n h(x_j)\frac{(-1)^{n-k}\sigma_{n-k}(x_0,\ldots,\hat{x_j},\ldots,x_n)}{(x_j-x_0)(x_j-x_1)\cdots(x_j-x_{j-1})(x_j-x_{j+1})\cdots(x_j-x_n)} $$
for $k=0,1,\ldots,n$. In the above formula $\sigma_{n-k}(x_0,\ldots,\hat{x_j},\ldots,x_n)$ is the $(n-k)$-th elementary symmetric function of $x_0,\ldots,\hat{x_j},\ldots,x_n$ where $x_j$ is omitted. Using the estimate $\sigma_{n-k}(x_0,\ldots,\hat{x_j},\ldots,x_n)\leq \binom{n}{n-k}2^{n-k}$ together with (\ref{eq16}) we get that, for every $k=0,1,\ldots,n$,  
\begin{eqnarray*}
|b_k|&\leq& \binom{n}{n-k}2^{n-k}n^n \frac{\alpha}{|E_\alpha|^n} \sum_{j=0}^n \frac{1}{j!(n-j)!}\\
&=&\binom{n}{n-k}2^{2n-k}\frac{n^n}{n!} \frac{\alpha}{|E_\alpha|^n}\,\,\, \leq\,\,\, c\frac{8^n }{\sqrt{n}}\frac{n^n}{n!}\frac{\alpha}{|E_\alpha|^n},
\end{eqnarray*}
where we used the estimate $\binom{n}{n-k}\leq \binom{n}{[\frac{n}{2}]} \leq c\frac{2^{n}}{\sqrt n}$. Hence
$$\max_{0\leq k \leq n}|b_k|\leq c  \frac{8^n }{\sqrt{n}}\frac{n^n}{n!}\frac{\alpha}{|E_\alpha|^n}	$$
and solving with respect to $|E_\alpha|$ we get
$$|E_\alpha| \leq c \bigg(\frac{\alpha}{\max_{0\leq k \leq n}|b_k|}\bigg)^\frac{1}{n}  .$$
\end{proof}

\bibliographystyle{amsplain}

\bibliography{Stein_Wainger_bibl}
\end{document}